\documentclass[12pt,reqno,twoside]{amsart}
\usepackage{amsthm,amsmath,amssymb,mathrsfs}
\usepackage{graphicx,epsfig,wrapfig,sidecap,subfig}
\usepackage[all]{xy}
\usepackage{proof}
\usepackage[usenames]{color}
\usepackage{hyperref}
\usepackage[capitalize]{cleveref}

\xyoption{dvips}
\xyoption{color}

\usepackage{color}

\title[Dense and Danzer]{Dense forests and Danzer sets}
\date{}
\author{Yaar Solomon}
\address{Stony Brook University, Stony Brook, NY {\tt yaar.solomon@stonybrook.edu}} 
\author{Barak Weiss}
\address{School of Mathematical Sciences, Tel Aviv University, Israel 
{\tt barakw@post.tau.ac.il}}

\newif\ifdraft\drafttrue
\draftfalse

\newcommand{\N}{{\mathbb{N}}}
\newcommand{\Z}{{\mathbb{Z}}}
\newcommand{\Q}{{\mathbb {Q}}}
\newcommand{\R}{{\mathbb{R}}}
\newcommand{\C}{{\mathbb{C}}}
\newcommand{\E}{{\mathbb{E}}}

\newcommand{\RA}{\mathscr{R}}

\newcommand{\conv}{{\rm conv}}

\newcommand{\diam}{{\rm diam}}

\newcommand{\spa}{{\rm span}}
\newcommand{\supp}{{\rm supp}}

\newcommand{\Ad}{{\operatorname{Ad}}}
\newcommand{\ASL}{\operatorname{ASL}}

\newcommand{\SL}{\operatorname{SL}}

\newcommand{\df}{{\, \stackrel{\mathrm{def}}{=}\, }}

\newcommand{\sm}{\smallsetminus}

\newcommand{\til}{\widetilde}
\newcommand{\vre}{\varepsilon}

\newcommand\Vol{\mathrm{vol}}

\newcommand {\ignore}[1]  {}

\newcommand\eq[2]{{\ifdraft{\ \tt [#1]}\else\ignorespaces\fi}\begin{equation}\label{#1}{#2}\end{equation}}
\newcommand {\equ}[1]{\eqref{#1}}

\newcommand{\absolute}[1] {\left|{#1}\right|}

\newcommand{\norm}[1]{\left\|{#1}\right\|}

\theoremstyle{plain}
\newtheorem{thm}{Theorem}[section]
\newtheorem{lem}[thm]{Lemma}
\newtheorem{prop}[thm]{Proposition}
\newtheorem{cor}[thm]{Corollary}

\theoremstyle{definition}
\newtheorem{definition}[thm]{Definition}
\newtheorem{question}[thm]{Question}

\newtheorem{example}[thm]{Example}

\newcommand{\CC}{\mathbb{C}}
\newcommand{\RR}{\mathbb{R}}
\newcommand{\QQ}{\mathbb{Q}}
\newcommand{\ZZ}{\mathbb{Z}}

\newcommand{\GC}{G^{\CC}}
\newcommand{\Liew}[1]{\lowercase{\mathfrak{#1}}}

\DeclareMathOperator{\SU}{SU}
\DeclareMathOperator{\Sp}{Sp}

\numberwithin{equation}{section}
\swapnumbers

\begin{document}

\begin{abstract}
A set $Y\subseteq\R^d$ that intersects every convex set of volume $1$
is called a Danzer set. It is not known whether there are Danzer sets
in $\R^d$ with growth rate $O(T^d)$. We prove that natural 
candidates, such as discrete sets that arise from substitutions and from
cut-and-project constructions, are not Danzer sets. For cut and
project sets our proof relies on the dynamics of 
homogeneous
flows. 
We
consider a weakening of the Danzer problem, the 
existence of a uniformly discrete dense forests, and we use
homogeneous dynamics (in particular Ratner's theorems on unipotent flows) to construct such sets. We also
prove an equivalence between the above problem and a well-known
combinatorial problem, and deduce the existence of  Danzer sets with
growth rate $O(T^d\log T)$, improving the previous bound of $O(T^d\log^{d-1}
T)$. 
\end{abstract}

\maketitle

\section{Introduction}\label{sec:intro}
This paper stems from a famous unsolved problem formulated by Danzer
in the 1960's (see e.g. \cite{GL87, croft, Gowers}). We will call a subset $Y
\subseteq \R^d$ a {\em Danzer set} if it intersects every convex
subset of volume 1. We will say that $Y$ {\em has growth $g(T)$}, where
$g(T)$ is some function, if 
\eq{eq: growth requirement}{
\# (Y \cap B(0,T)) = O(g(T))
} (as usual $f(x)=O(g(x))$
means $\limsup_{x \to \infty} \frac{f(x)}{g(x)} < \infty$ and
$B(0,T)$ is the Euclidean ball of radius $T$ centered at the origin in
$\R^d$). Danzer asked whether for $d \geq 2$ there is a Danzer set with growth
$T^d$. In this paper we present several results related to  
this question. 

The only prior results on Danzer's question of which we are aware are due to Bambah and
Woods. Their paper \cite{BW71} contains two results. The first is a
construction of a Danzer set in $\R^d$ with growth rate
$T^d\log^{d-1}(T)$, and the second is a proof that any finite union
of grids\footnote{A \emph{grid} is a translated lattice.} is not a
Danzer set. Our paper contains parallel results. 
 
We prove the following theorems. For detailed definitions of the terms
appearing in the statements, we refer the reader to the section in
which 
the result
is proved.  
\begin{thm}\label{thm:Substitution_Not_Danzer}
Let $H$ be a primitive substitution system on the polygonal basic tiles
$\{T_1,\ldots,T_n\}$ in $\R^d$. Any Delone set, which is obtained
from a tiling $\tau\in X_H$ by picking a point in the same location in
each of the basic tiles, 
 is not a Danzer set. Also the set of vertices of tiles in such a
 tiling is not a Danzer set.  
\end{thm}

In particular the vertex set of a Penrose
tiling is not a Danzer set. The vertex set of a
Penrose tiling has another description, namely as a cut-and-project
set. We now consider such sets.  

\begin{thm}\label{thm:Ratner,cut-and-project}
Let $\Lambda$ be a finite union of cut-and-project sets. Then $\Lambda$ is not a Danzer set. 
\end{thm}

As for positive results, one may try to construct sets which either
satisfy a weakening of the Danzer condition, or a weaker growth
condition. The following results are in this vein. 
A set $Y\subseteq\R^d$ is called a \emph{dense forest} if
there is a function
$\varepsilon(T)\xrightarrow{T\to\infty}0$ such that for
any $x\in\R^d$ and any direction $v\in\mathbb{S}^{d-1} \df \{v \in
\R^d: \|v\| = 1\}$, the distance from $Y$ to the line segment of
length $T$ going from $x$ in direction $v$ is at most
$\varepsilon(T)$. It is not hard to show that a Danzer set is a dense
forest. 

\begin{thm}\label{thm:dense_forest}
Let $U \cong \R^d$ and suppose $X$ is a compact metric space on which
$U$ acts continuously and completely uniquely ergodically. Then for
any cross-section $\mathcal{S}$ and any $x_0 \in X$, the set of `visit
times'  
\[\mathcal{D} \df \{u \in U: u.x_0 \in \mathcal{S}\}\]
is a uniformly discrete set which is a dense forest. In particular,
uniformly discrete dense forests exist in $\R^d$ for any $d$. 
\end{thm}

By {\em completely uniquely ergodically} we mean that the restriction
of the action to any one-parameter subgroup of $U$ is uniquely
ergodic. Our construction of completely uniquely ergodic actions 
relies on Ratner's theorem
and results on the structure of lattices in algebraic groups. 

In order to construct Danzer sets which grow slightly faster than
$O(T^d)$, we first establish an equivalence between this question and
a related finitary question, namely the
`Danzer-Rogers question' (see Question \ref{ques:CG_Danzer}). 

\begin{thm}\label{thm:Equivalent_Questions}
For a fixed $d\ge 2$, and a function $g(x)$ of polynomial growth, the following are equivalent:
\begin{itemize}
\item[(i)]
There exists a Danzer set $Y\subseteq\R^d$ of growth $g(T)$.
\item[(ii)] 
For every $\varepsilon>0$ there
exists  $N_\varepsilon\subseteq[0,1]^d$, such that $\# N_{\varepsilon}
= O(g(\varepsilon^{-1/d}))$, and such that $N_{\vre}$ intersects every rectangle of volume
$\varepsilon$ in $[0,1]^d$. 
\end{itemize}
\end{thm}

\begin{cor}\label{cor:rational_Danzer}
If $D\subseteq\R^d$ is a Danzer set of growth rate $g(T)$, where
$g(x)$ has polynomial growth, then there exists a Danzer set contained
in 
$\Q^d$ of growth rate $g(T)$.  
\end{cor}

Using Theorem \ref{thm:Equivalent_Questions} and known results
for the Danzer-Rogers question, we obtain:

\begin{thm}\label{thm:T^dlogT}
There exists a Danzer set in $\R^d$ of growth rate $T^d\log T$.
\end{thm}

Note that for all $d \geq 3$, this improves the result of \cite{BW71}
mentioned above and represents the slowest known growth rate for 
a Danzer set. 
\subsection{Structure of the paper}
We have attempted to keep the different sections of this paper
self-contained. The material on substitution tilings and
the cut-and-project sets, in particular the proofs of Theorems
\ref{thm:Substitution_Not_Danzer} and
\ref{thm:Ratner,cut-and-project}, are contained in \S\ref{sec:substitution} and
\S\ref{sec:cut-and-project} respectively. More results from
homogeneous flows are used in order to prove Theorem \ref{thm:dense_forest} in
\S\ref{sec:dense_forest}. In
\S\ref{sec:equi_combinatoric_question} we introduce some terminology 
from computational geometry and prove Theorem
\ref{thm:Equivalent_Questions} and Corollary
\ref{cor:rational_Danzer}. More background from computational geometry
and the proof of Theorem
\ref{thm:T^dlogT} are in
\S\ref{sec:T^dlogT_construction}. In \S \ref{sec: questions} we list
some open questions related to the Danzer problem. 

\subsection{Acknowledgements} The proof of Theorem
\ref{thm:Ratner,cut-and-project} given here relies on a suggestion of
Andreas Str\"ombergsson. Our initial strategy required a detailed analysis of
lattices in algebraic groups satisfying some conditions, and an appeal
to Ratner's theorem on orbit-closures for homogeneous flows. We reducd
the problem to a question on algebraic groups which we were
unable to solve ourselves and after 
consulting with several experts, we received a complete answer from 
Dave Morris, 
and his argument appeared in an appendix of the original
version of this paper. Later Str\"ombergsson gave us a simple
argument which made it possible to avoid the results of Morris and to
avoid Ratner's theorem. The proof which appears here is 
Str\"ombergsson's and we are grateful to him for agreeing to include it, 
and to Dave Morris for his 
earlier proof. 
We are also grateful to  Manfred
Einsiedler, Jens Marklof, Tom Meyerovitch, Andrei Rapinchuk, Saurabh Ray, Uri Shapira and Shakhar
Smorodinsky  for useful
discussions. Finally, we are grateful to Michael Boshernitzan
for telling us about Danzer's question.  We acknowledge the support of
ERC starter grant DLGAPS 279893.

\section{Nets that Arise from Substitution Tilings}\label{sec:substitution}
In this section we prove Theorem
\ref{thm:Substitution_Not_Danzer}, i.e. that primitive substitution
tilings do not give rise to Danzer sets. We begin by quickly recalling the
basics of the theory of substitution tilings. For further reading we
refer to \cite{GS87,Ra99,Ro04,So97}. 

\subsection{Background on substitutions}\label{subsec:substitution_basics}
A \emph{tiling} of $\R^d$ is a countable collection of tiles $\{S_i\}$
in $\R^d$, each of which is the closure of its interior, such that
tiles intersect
only in their boundaries, and with $\bigcup_iS_i=\R^d$.
We say that the tiling is 
\emph{polygonal} if the tiles are
$d$-dimensional polytopes, i.e. convex bounded sets that can be
obtained as an intersection of finitely many half-spaces. All tilings
considered below are polygonal. 


Given a finite collection of tiles $\mathcal{F}=\{T_1,\ldots,T_n\}$ in
$\R^d$, a \emph{substitution} is a map $H$ that assigns to every
$T_i$ a tiling of the set $T_i$ by isometric copies of $\zeta
T_1,\ldots,\zeta T_n$, where $\zeta\in(0,1)$ is fixed and does not
depend on $i$. The definition of $H$ extends in an obvious way to
replace a finite union of isometric copies of the $T_i$'s by isometric
copies of the $\zeta T_j$'s. By applying $H$ repeatedly, and rescaling
the small tiles 
back to their original sizes, we tile larger and larger regions of the
space. \emph{Substitution tilings} are tilings of $\R^d$ that are
obtained as limits of those finite tilings. More precisely, set
$\xi=\zeta^{-1}>1$ and define  
\[\mathcal{P}=\{(\xi H)^m(T_i):m\in\N, i\in\{1,\ldots,n\}\}.\]
The \emph{substitution tiling space} $X_H$ is the set of tilings
$\tau$ of $\R^d$ having the property that for every compact set
$K\subseteq\R^d$, there exists some $P\in\mathcal{P}$, such that the
patch defined by $\{\text{tiles }T\text{ of }\tau: T\subseteq K\}$ is
a sub-patch of $P$. The tilings $\tau\in X_H$ are \emph{substitution tilings}
that correspond to $H$, and the constant $\xi$ is referred to as the
\emph{inflation constant} of $H$. For every $i$ the isometric copies
of $T_i$ are called \emph{tiles of type $i$}. 
 
A substitution $H$ on $\mathcal{F}=\{T_1,\ldots,T_n\}$ defines the
\emph{substitution matrix}, which is a non-negative integer matrix
$A_H=(a_{ij})$, where $a_{ij}$ is the number of appearances of
isometric copies of $\zeta T_i$ in $H(T_j)$. $H$ is called
\emph{primitive} if $A_H$ is a primitive matrix. Namely, $A_H^m$ has
strictly positive entries, for some $m\in\N$. Observe that primitivity
is a natural assumption in this context, since otherwise we could get
a smaller substitution system by restricting $H$ to a subset of
$\{T_1,\ldots,T_n\}$ (and possibly replacing $H$ by some fixed power of
$H$). For example, the matrix
$A_H=\begin{pmatrix}2&1\\1&1\end{pmatrix}$ corresponds to the
substitution of the Penrose triangles that are presented below in
                                Figure
                                \ref{Figure:Penrose_substitution}. 

\begin{figure}[ht!]
\begin{center}
\includegraphics{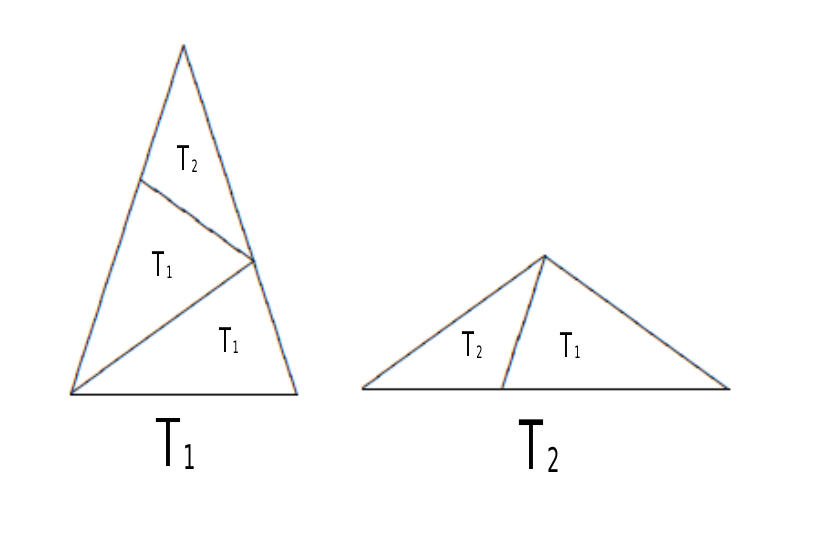}
\caption{Penrose substitution rule.}
\label{Figure:Penrose_substitution}
\end{center}
\end{figure}

Primitive substitution tiling satisfy a useful `inflation' property 
(see
\cite{Ro04}). For $m\in\N$ it is convenient to consider the set of
inflated tiles $\mathcal{F}^{(m)}=\{\xi^mT_1,\ldots,\xi^mT_n\}$ with
the dissection rule $H^{(m)}$ that is induced by $H$, and the
substitution tiling space $X_{H^{(m)}}$. 
\begin{prop}\label{prop:tau_m_Def}
If $H$ is a primitive substitution then $X_H\neq\emptyset$ and for
every $\tau\in X_H$ and $m\in\N$ there exists a tiling $\tau_m\in
X_{H^{(m)}}$ that satisfies $H^m(\tau_m)=\tau$.    
\end{prop} 

We use the following terminology throughout the proof of Theorem
\ref{thm:Substitution_Not_Danzer}. Let $H$ be a primitive substitution
defined on a finite set of polygonal tiles
$\mathcal{F}=\{T_1,\ldots,T_n\}$ in $\R^d$, with an inflation constant
$\xi>1$. A \emph{choice function} is a function which assigns to each
tile $T_i$ a point $h(T_i)\in T_i$. For a
tiling $\tau_0\in X_H$, and a choice function $h$, denote by
$Y_{\tau_0,h}$ the Delone set\footnote{A {\em Delone set} or {\em
    separated net} is a set $Y \subseteq \R^d$ satisfying $\inf_{x,y \in
    Y, x \neq y} \|x-y\| >0,\ \sup_{x \in \R^d} \inf_{y \in Y}
  \|x-y\|<\infty$.}
that is obtained by placing one point
in each tile of $\tau_0$, with respect to the choices of $h$. More
precisely, each tile of $\tau_0$ is equal to $g(T_i)$, for some $i$
and some isometry $g$ of $\R^d$. The function $h$ can naturally be
extended to all the tiles of $\tau_0$, then to collections of tiles of
$\tau_0$, and in turn to tilings $\tau\in X_H$. We define
$Y_{\tau_0,h}=h(\tau_0)$. Theorem \ref{thm:Substitution_Not_Danzer}
says that $Y_{\tau_0,h}$ is not a Danzer set, for any primitive
substitution polygonal tiling in $X_H$ and any choice function $h$. 

In this section we will use the letter $D$ to denote the Euclidean
distance. For a closed set $A$ and a point $x \in \R^d$, we will
write $D(x,A) = \sup_{a \in A} D(x,a)$, and for a set $A$ and $\delta>0$ we
denote by $U_\delta(A)$ the $\delta$-neighborhood of $A$. Also let $V_d$ denote the volume of the
$d$-dimensional unit ball.

\begin{proof}[Proof of Theorem \ref{thm:Substitution_Not_Danzer}]
Let $h$ be a choice function and let $Y=Y_{\tau_0. h}$. We
first consider the case where $h(T_i)\in int(T_i)$ for every
$i$. Denote by  
\begin{equation}\label{eq:delta_substitution}
\delta=\min_i\{D(h(T_i),\partial T_i)\},
\end{equation}
and note that we are assuming for the moment that $\delta>0$. Denote
by $\partial\tau$ the union 
of all the boundaries of tiles of a tiling $\tau$. Then our
definition of $\delta$ ensures that any element $x \in
Y$ satisfies $D(x, \partial \tau_0) \geq \delta$. 

If a $d-1$-dimensional face of a tile in $\tau_0$ contains
a segment of length $t$, then the same face of the same type of tile
in $\tau_m$ contains a segment of length $t\cdot\xi^m$. The tiles are
polygonal, so let $m$ be large enough such that some face $F$ of some
tile in $\tau_m$ contains a segment $L$ of length
$\ell$ where $\ell \delta^{d-1}V_{d-1}>1$. Since 
$\partial\tau_m\subseteq\partial\tau_0$, $L$ is also contained in
$\partial\tau_0$. By (\ref{eq:delta_substitution}), $U_\delta(L)$
misses $Y$. Clearly $U_{\delta}(L)$ is convex, and the
choice of $m$ and $\ell$ guarantees that the 
volume of $U_{\delta}(L)$ is at least 1.  

Now suppose $\delta=0$, i.e. for some $i$, $h(T_i) \in \partial
T_i$. Choose
$\delta_1$ small enough so that the following hold: if $h(T_i)
\notin \partial T_i$ then $2 \delta_1 < D(h(T_i), \partial
T_i)$; if $h(T_i) \in \partial T_i$ then $h(T_i)$ is contained in some
of the boundary faces of $T_i$, and we require that $2\delta_1$ is smaller than
the distance from $h(T_i)$ to the boundary faces of $T_i$ which do not
contain $h(T_i).$ With this choice of $\delta_1$, let $L
\subseteq \partial \tau_m$ be a line as in the preceding case, where
$m$ is chosen large enough so that the length $\ell$ of $L$ satisfies  $\ell
\delta_1^{d-1}V_{d-1}>1$. Let $v$ be a vector of length $\delta_1$
which is 
perpendicular to the boundary face containing $L$ and let $L' = L+v$. Then
$U'=U_{\delta_1}(L')$ is contained in $U_{2\delta_1}(L)$ and thus
contains none of the points of $Y$ which are not in
$\partial \tau_0$. Moreover $L'$ is of distance $\delta_1$ from the
boundary faces containing $L$, so $U'$ is disjoint from the boundary
faces but every point of $U'$ is within distance $2 \delta_1$ from
these boundary faces. Our definition of $\delta_1$ ensures that $U'$ does not
contain points belonging to $Y \cap \partial \tau_0$.  Thus $U'$
misses $Y$, and is a convex set of volume at least 1, as required.

Finally consider the set of vertices $Y$ of the tiling $\tau_0$. Let $\delta_1$ be
small enough so that for any boundary face of any $T_i$, the distance
from the face to any of the vertices not on the face is greater than
$2\delta_1$. Using the same $m$ and the same 
line $L \subseteq \partial \tau_0$ as in the preceding paragraph, define $L', U'=U_{\delta_1}(L')$ as
above. Then the fact that $L \subseteq \partial
\tau_0$ and the definition of $\delta_1$ ensures that $U' \cap Y =
\varnothing,$ so $Y$ is not a Danzer set. 
\end{proof}

\section{Cut-and-project sets}\label{sec:cut-and-project}
Let $d, k, n$ be integers with $d>1, k\geq 1$ and $n = d+k$, and write
$\R^n$ as the direct sum of $\R^d$ and $\R^k$. We refer to the numbers
$d, k, n$ as the \emph{physical dimension, internal dimension}, and
\emph{full dimension}, the spaces $\R^d$ and $\R^k$ are the
\emph{physical and internal spaces}, and denote by $\pi_{phys}:\R^n
\to \R^d$ and $\pi_{int}: \R^n \to \R^k$ the projections associated
with this direct sum decomposition; i.e. for $\vec{x}=(x_1, \ldots,
x_n)$,  
\[ \pi_{phys}(\vec{x}) = (x_1, \ldots, x_d), \ \ \pi_{int}(\vec{x}) =
(x_{d+1}, \ldots, x_n).\]
Let $L \subseteq\R^n$ be a grid (recall that a grid is a translate of
a lattice) and let $W \subseteq \R^k$ be a bounded subset. The set 
\[ \Lambda(L, W) \df \pi_{phys} \left(L \cap   \pi_{int}^{-1}(W)\right)\]
is called a
\emph{cut-and-project set} or \emph{model set}. Such sets have been
extensively studied (see \cite{Me94, BM00, Se95} and the references
therein). In particular the vertex set of a Penrose tiling provides an
example of a cut-and-project set. In this section, following
\cite{MS13}, we will apply homogeneous dynamics in order to analyze
the geometry of cut-and-project sets, and to prove Theorem
\ref{thm:Ratner,cut-and-project}. 

We begin with a dynamical characterization of Danzer sets. It will be
more convenient to work with the class of dilates of Danzer sets. We
say that $Y \subseteq \R^d$ is a {\em Dilate of a Danzer set} (or {\em
  DDanzer} for short), if there is $c>0$ such that the dilate $cY =
\{cy: y \in Y\}$ is Danzer. Let
$\ASL_d(\R)\cong \SL_d(\R)\ltimes\R^d$ denote the group of affine
orientation-preserving, measure-preserving transformations of $\R^d$. Since $\ASL_d(\R)$ maps
convex sets to convex sets and preserves their volumes, the property
of being DDanzer set is invariant under the action of $\ASL_d(\R)$ on
subsets of $\R^d$. Moreover it can be characterized in terms of this
action.

\begin{prop}\label{prop:dynamical}
Let $Y \subseteq \R^d$. Then $Y$ is DDanzer if and only if there is $T>0$ such 
that for every $g \in \ASL_d(\R), \, gY \cap B(0,T) \neq \varnothing.$ 
\end{prop}
\begin{proof}
This is a straightforward corollary of John's theorem 
(see \cite[Lecture 3]{Ba97}) that asserts that any convex subset $K$
of $\R^d$ contains an ellipsoid $E$ with $E \subseteq K \subseteq
dE$ (where $dE$ denotes the dilation of $E$ by a factor of $d$). Since
all ellipsoids of volume $\Vol(B(0,T))$ map to the ball 
$B(0,T)$ under an affine transformation, the result follows.  
\end{proof}

We begin with a few reductions of the problem. It will simplify notation to take $\R^2$ and in fact this entails no loss of generality: 

\begin{prop}\label{prop:reduction_d=2}
Suppose $d \geq 3$. If $\Lambda \subseteq \R^d$ is a union of
cut-and-project sets 
(respectively, a DDanzer set), $\R^d = \R^2 \oplus
\R^{d-2}$ is a direct sum decomposition with associated projections
$\pi_1: \R^d \to \R^2, \pi_2 : \R^d \to \R^{d-2}$, and $W \subseteq
\R^{d-2}$ is bounded, then $\Lambda' \df 
 \{\pi_1(x): x \in \Lambda \cap \pi_2^{-1}(W)\}$ is also a union of
 cut-and-project sets (respectively, a DDanzer set).  
\end{prop}
\begin{proof}
Left to the reader. 
\end{proof}

From this point on we consider $\R^n$ with its standard $\Q$-structure
$\Q^n$.
We will make a convenient reduction. We say that a
lattice $L \subseteq \R^n$ is \emph{$\Q$-irreducible with respect to
 the physical space} if there is no proper
rational subspace $R \subseteq \R^n$ such that $R \cap L$ is a lattice
in $R$ and $R$ contains $\ker \pi_{int}$ (the second condition may  be
stated equivalently by saying that $R$ contains the physical space
$\R^d$). We say that a grid $L$ is $\Q$-irreducible with respect to the
physical space if the underlying lattice $L-L$ is, and we will say that a
cut-and-project set $\Lambda$ is {\em 
  irreducible} if it can be presented as $\Lambda(L,W)$ for some
grid $L$ and some window $W$, so that $L$ is $\Q$-irreducible with
respect to the physical space. 

\begin{prop}\label{prop:Q_irreducible}
Suppose $\Lambda$ is a finite union of cut-and-project sets. Then
there is a finite union of cut-and-project sets
$\Lambda' = \Lambda'_1 \cup
\cdots \cup \Lambda'_s$ containing $\Lambda$, such that each
$\Lambda'_j$ is irreducible. 
\end{prop}

\begin{proof}
By induction on the number of cut-and-project sets defining $\Lambda$,
we can assume that this number is one, i.e. $\Lambda$ is a
cut-and-project set. Also by induction on the internal dimension 
it suffices to show that if $\Lambda$ is not
irreducible, and the internal dimension is $k$, then there is a finite union of
cut-and-project sets containing $\Lambda$ with internal dimension smaller than $k$.

Our notation is as follows: $L$ is the grid, $L_0 = L-L$ is the
underlying lattice, 
$\pi_{phys}, \pi_{int}$ are the projections, $V_{phys} = \ker
\pi_{int}$ is the physical space, $V_{int} = \ker \pi_{phys}$ is the
internal space, and $W \subseteq V_{int}$ is the window,  in the construction of
$\Lambda$. Let $R$ be a proper rational subspace of $\R^n$ such that
$L' \df L_0 \cap R$ is a lattice in $R$, and $R$ contains $V_{phys}$.
Let $R' = R \cap V_{int}$. 
We will find some positive integer $s$ and for $j=1,
\ldots, s$ find bounded sets $W'_j \subseteq
R'$, and vectors $y_j \in R$, satisfying the following: 
for any $\ell \in L$ such that $\pi_{int}(\ell) \in W$ there is $j
\in \{1, \ldots, s\}$ and 
$\til \ell \in L'_j \df  L'+y_j$ such that $\pi_{phys}(\til \ell) =
\pi_{phys}(\ell)$ and $\pi_{int}(\til \ell) \in W'_j$. 
This will complete the proof, since it implies that $\Lambda =
\pi_{phys}(L \cap \pi_{int}^{-1} (W))$ is contained in $\bigcup
\Lambda'_j$, where $\Lambda'_j \df \pi_{phys}(L'_j \cap \pi_{int}^{-1}(W'_j))$. 

Let $V_0
\subseteq V_{int}$ be a subspace such that $R \oplus V_0 =
\R^n$, let $\pi, \pi_0$ be the projections associated with this direct
sum decomposition. Since $V_0 \subseteq V_{int}$ we have $\pi_0 = \pi_{int}
\circ \pi_0$ and the assumption that $V_{phys} \subseteq R$ implies
$\pi_0 = \pi_0 \circ \pi_{int}$. Since $L'$ is a lattice in $R$,  
$\pi_0(L)$ is discrete in $V_0$. Since $W$ is bounded, so is
$\pi_0(W)$, and hence $\pi_0(L) \cap \pi_0(W)$ is finite. Let $\ell_1, \ldots, \ell_s \in L$ such that 
$$\pi_0(L)
\cap \pi_0(W) = \{\pi_0(\ell_j): j=1, \ldots, s\}.$$

For $j=1, \ldots, s$ write $\ell_j = x_j+y_j$ where $x_j =
\pi_0(\ell_j)$ and $y_j = \pi(\ell_j),$ and let $W'_j = R \cap  (W
-x_j)$. Clearly $W'_j \subseteq R$, and since $W$ is bounded, 
so is $W'_j$. 

Suppose $\ell \in L$ with $\pi_{int}(\ell) \in W$. We can write $\ell=
\pi_0(\ell)+\til \ell$ with $\til \ell
\in R$. Then $\pi_{phys}(\til \ell) = \pi_{phys}(\ell)$ since $\ell -
\til \ell \in V_0 \subseteq V_{int}$. Since $\pi_0 =
\pi_0 \circ \pi_{int}$, we have $\pi_0(\ell) \in \pi_0(W)$, and there is $j$ so that
$\pi_0(\ell)=\pi_0(\ell_j) = x_j$, and 
$$
\til \ell - y_j = \ell-\pi_0(\ell) -(\ell_j - x_j) = \ell - \ell_j \in
L \cap R. 
$$
This shows $\til \ell \in L'_j.$ 
Finally to see that $\pi_{int}(\til \ell) \in W'_j$, note that 
$$\pi_{int}(\til \ell) = \pi_{int}(\ell) - \pi_{int}\circ \pi_0(\ell) =
\pi_{int}(\ell) - \pi_{0}(\ell) = \pi_{int}(\ell) - x_j \in W -
x_j,$$
and $\pi_{int}(\til \ell) \in R$ since $\pi_0 \circ \pi_{int}(\til
\ell) = \pi_0(\til \ell)=0.$
\end{proof}

Let $\mathcal{X}_n$ denote the space of
unimodular lattices in $\R^n$. Recall that this space is identified
with the quotient $\SL_n(\R)/\SL_n(\Z)$ via the map which sends the
coset $g\SL_n(\Z)$ to the lattice $g\Z^n$, and that this
identification intertwines the action of $\SL_n(\R)$ on
$\mathcal{X}_n$ by linear transformations of $\R^n$, with the
homogeneous left-action $g_1\tau(g_2) = \tau(g_1g_2),$ where $\tau:
\SL_n(\R) \to \mathcal{X}_n$ is the natural projection. 
A crucial ingredient in our argument will be the following fact:

\begin{prop}[Andreas Str\"ombergsson]\label{prop: strombergsson}
Let $H_0 \cong \SL_2(\R)$ be embedded in $\SL_n(\R)$ in the upper left-hand
corner; i.e., with respect to the decomposition $\R^n = \R^2 \oplus
\R^{n-2}$, $H_0$ acts via its standard action on the first
summand and trivially on the second summand. Then for any $x \in
\mathcal{X}_n$, the orbit $H_0x$ is unbounded (i.e. its closure
is not
compact). 
\end{prop}

\begin{proof}
Suppose by contradiction that $H_0x$ is 
bounded, and let $g_0 \in \SL_n(\R)$ such that $x = \tau(g_0)$. By
Mahler's compactness criterion (see e.g. \cite[Chap. 3]{GL87}), there is $\vre>0$ such that 
\eq{eq: for a contradiction}{
\text{for any } v \in g_0\Z^n \sm \{0\} \text{ and any }h \in H_0, \, \|hv\|\geq
\vre,
}
where $\|\cdot \|$ denotes the sup norm on $\R^n$. Let $P, Q$
denote respectively the projections $\R^n \to \R^2, \, \R^n \to
\R^{n-2}$ corresponding to the direct sum decomposition $\R^n \cong
\R^2 \oplus \R^{n-2}.$ Then $H_0$ acts transitively on
all nonzero vectors in the image of $P$, and acts trivially on the
image of $Q$; that is for any $v \in \R^n$ and any $h \in H_0$, $P(hv) =
hP(v), Q(hv)=Q(v)$. Also since we have chosen the sup norm we can
write $\|v 
\| = \max (\|P(v)\|, \|Q(v)\|)$ for any $v \in \R^n$. Since $g_0\Z^n$
is an abelian group of rank $n$, either $g_0\Z^n \cap \ker \, Q \neq \{0\}$, or $Q(g_0\Z^n)$ is not
discrete in $\R^{n-2}$. In either case there is $v_0 \in g_0\Z^n \sm \{0\}$ such that
$\|Q(v_0)\|<\vre$. Since $H_0$ acts transitively on $\R^2  \sm
\{0\}$, there is $h \in H_0$ such that $\|hP(v_0)\|< \vre$. This implies
that $$\|hv_0\| = \max(\|P(hv_0)\|, \|Q(hv_0)\|) = \max(\|hP(v_0)\|,
\|Q(v_0)\|) < \vre, $$
a contradiction to 
\equ{eq: for a contradiction}. 
\end{proof}

The following statement was proved in
response to our question by Dave Morris. The proof of Morris relied on
structure theory for algebraic groups and appears in the appendix to a preliminary
version of our paper. Here we give a simpler proof
using Proposition \ref{prop: strombergsson}.

\begin{cor}[Dave Morris]\label{thm:Morris}
Let $H_0 \cong \SL_2(\R)$ be as above. 
Then for any semisimple
group $G \subseteq \SL_n(\R)$ containing $H_0$, there is no
conjugate $G'$ of $G$ in $\SL_n(\R)$ whose intersection with
$\SL_n(\Z)$  
is a cocompact lattice in $G'$. 
\end{cor}

\begin{proof}
If such a conjugate $G' = g_0^{-1} Gg_0$ existed, then $G'\Z^n$ would
be compact, and hence so would the orbit-closure
$\overline{H_0\tau(g_0)} \subset G\tau(g_0) \cong g_0^{-1} G'\Z^n$,
contradicting Proposition \ref{prop:
  strombergsson}. 
\end{proof}

Let $\mathcal{Y}_n \df \ASL_n(\R)/\ASL_n(\Z)$ be the space of
$n$-dimensional unimodular grids, and let $H \df \ASL_2(\R)$ be
embedded in $\ASL_n(\R)$ in the upper left-hand corner. Equivalently,
the action of $H$ on $\R^n$ preserves the physical space and acts on
it as the group of affine volume preserving
transformations. Proposition \ref{prop: strombergsson} restricts the
orbit-closures of $H$ acting on 
$\mathcal{Y}_n$. 

\begin{cor}\label{cor:from_witte1}
Let $H$ be embedded as above. Then for any $y \in \mathcal{Y}_n,$ the
orbit-closure $\overline{Hy}$ is noncompact.  
\end{cor}

\begin{proof}
Let $P: \ASL_n(\R) \to \SL_n(\R)$ be the natural projection sending an
affine map to its linear part. 
Then $P$ is defined over $\Q$, and induces an equivariant map
$\bar{P}: \mathcal{Y}_n \to \mathcal{X}_n$, which has a compact
fiber. In particular $\bar{P}$ is a proper map. We denote $H_0 \df
P(H) \cong \SL_2(\R), \, x = \bar{P}(y)$. 
By Proposition \ref{prop: strombergsson}, 
 $\bar{P}\left(\overline{Hy}\right)
= \overline{H_0x}$ is not compact, and hence $\overline{Hy}$ is not
compact.
\end{proof}

\begin{cor}\label{cor:all_translations}
Keeping the notations and assumptions of Corollary
\ref{cor:from_witte1}, assume that the linear part of the grid $y$ is
$\Q$-irreducible with respect to the physical space. Then the
orbit-closure $\overline{Hy}$ is invariant under all translations in
$\R^n$.  
\end{cor}
\begin{proof}
We keep the notations of the previous proof. 
Let $T
\cong \R^n$ be the unipotent radical of $\ASL_n(\R)$, i.e. the normal
subgroup of affine maps which are actually translations. 
We need to show that $\Omega \df \overline{Hy}$ is
$T$-invariant for any $y \in \mathcal{Y}_n$. Let $S \subseteq T, \, S \cong \R^2$ be the
group of translations in the direction of the physical which act
trivially on the internal space. 
The assumption that $y$ is $\Q$-irreducible with respect to the physical
space implies that there is no intermediate linear subspace $S
\subseteq S' \varsubsetneq T$ such that $S'y$ is closed, and this implies
that $\overline{Sy} = Ty.$ Since $T$ is
normal in $\ASL_n(\R)$, for any $h \in H$, $\Omega \supseteq
\overline{hSy} = hTy = Thy$, i.e. there is a dense set of $z \in
\Omega$ for which $Tz \subseteq \Omega$. This implies that $\Omega$ is
$T$-invariant. 
\end{proof}

We will need similar statements for products of spaces
$\mathcal{Y}_{n_i}$. 
If $z_{\ell}$ is a sequence in a topological space,  we will write
$z_\ell\xrightarrow{\ell\to\infty} \infty$ if the sequence has no
convergent 
subsequence.

\begin{prop}\label{prop:general_fact}
Suppose $\mathcal{Z}_1, \ldots, \mathcal{Z}_r$ are locally compact
spaces, $H$ is a topological group acting continuously  
on each $\mathcal{Z}_i$, such that for every $i$  and every $z \in
\mathcal{Z}_i$ there is a sequence $(h_j) \subseteq H$ for which $h_jz
\xrightarrow{j\to\infty} \infty$. Then for every $(z_1, \ldots, z_r)
\in\mathcal{Z}_1 \times \cdots \times \mathcal{Z}_r$ there is a
sequence $(h_j) \subseteq H$ such that for each $i$, $h_jz_i
\xrightarrow{j\to\infty}\infty$. 
\end{prop}

\begin{proof}
By induction on $r$. If $r=1$ this is immediate from assumption, and
we suppose $r \geq 2$ and $(z_1, \ldots, z_r) \in\mathcal{Z}_1 \times
\cdots \times \mathcal{Z}_r$. By the induction hypothesis there is a
sequence $(g_j) \subseteq H$ such that $g_j
z_i\xrightarrow{j\to\infty}\infty$ in $\mathcal{Z}_i$ for $i=1,
\ldots, r-1$. If $g_j z_r \xrightarrow{j\to\infty} \infty$ in
$\mathcal{Z}_r$ then there is nothing to prove. Otherwise we may
replace $g_j$ by a subsequence to assume that $g_jz_r \to z$. By
assumption there is a sequence $(h'_j) \subseteq H$ such that $h'_j z
\xrightarrow{j\to\infty} \infty$ in $\mathcal{Z}_r$. Our required
subsequence will be obtained by replacing $(g_j)$ with a subsequence
and selecting $h_j = h'_j g_j$.  

To this end, by induction on $i_0=1, \ldots, r$, we will choose a
subsequence  of $(g_j)$ (which we continue to denote by $(g_j)$) with the following property:  
\eq{eq: prove this by induction}{
h'_j \til g_j z_i \xrightarrow{j\to\infty}\infty,  \text{ for every
subsequence } (\til g_j) \subseteq (g_j) \text{ and } i <i_0.
}
The base of the induction corresponds to the case $i_0=1$ in which
case \equ{eq: prove this by induction} is vacuously satisfied. Let
$\{K_\ell: \ell \in \N\}$ be an exhaustion of $\mathcal{Z}_{i_0}$ by
compact sets. Suppose first that $i_0 < r$. Then for any $j$ there is
$\ell$ so that if $z \notin K_{\ell}$ then $h'_jz \notin K_j$. This
implies that there is  $j_0 = j_0(j)$ such that for all $j' \geq j_0$,
$h'_j g_{j'} \notin K_j$. Therefore if we replace 
$(g_j)$ by the subsequence $(g_{j_0(j)})$ then \equ{eq: prove this by
  induction} will hold. Finally if $i_0 = r$ then since $g_j z_r \to
z$ and 
$h'_{j} z \xrightarrow{j\to\infty} \infty$, for each $j$ we can find
$j_0 = j_0(j)$ such that for $j' \geq j_0$, $h'_j g_{j'}z_r \notin
K_j$ and so,  if we replace $(g_j)$ by $(g_{j_0(j)})$ then \equ{eq:
  prove this by induction} will hold. 
\end{proof}

We will need a similar but stronger statement for the case of
translations on vector spaces.  
 
\begin{prop}\label{prop:similar_for_subspaces}
Suppose $V_1, \ldots, V_r$ are vector spaces, $P_i : V_1 \times
\cdots\times V_r \to V_i$ is the natural projection, $Q_i \subseteq V_i$
is a hyperplane for each $i$, and $\bar{P}_i$ is the composition of
$P_i$ with the quotient map $V_i \to V_i/Q_i$.  If $U \subseteq V_1
\times\cdots \times V_r$ is a linear subspace such that $P_i(U)=V_i$
for each $i$, then there is a sequence $(u_j) \subseteq U$ such that for
each $i$, $\bar{P}_i(u_j) \xrightarrow{j\to\infty} \infty$ in
$V_i/Q_i$. 
\end{prop}

\begin{proof}
In view of the surjectivity of $P_i|_U$, the preimage $U_i \df U \cap 
P_i^{-1}(Q_i)$ is a hyperplane of $U$. Let $d$ be a
translation-invariant metric on $U$. We may take any sequence $(u_j)
\subseteq U$ such that $d(u_j, \bigcup_i U_i) \xrightarrow{j\to\infty}
\infty$.  
\end{proof}

\ignore{
Recall that two lattices $L_1, L_2$ in $\R^n$ are \emph{commensurable}
if $L_1 \cap L_2$ is of finite index in both.  

\begin{prop}\label{prop:not_commensurable}
Suppose $\Lambda = \Lambda_1 \cup \cdots \cup \Lambda_r$ is a union of
cut-and-projects sets and is a DDanzer set. Write
$\Lambda_i=\Lambda(L_i, W_i)$, where $L_i \subseteq \R^{n_i}$. Suppose
that $n_1 = n_2$ and $L_1, L_2$ are commensurable. Then there is a set
$\Lambda' = \Lambda'_1 \cup \cdots \cup \Lambda'_{r-1}$ which is also
DDanzer and a union of fewer cut-and-project sets. 

\end{prop}
\begin{proof}
If $L_1, L_2$ are commensurable then the subgroup $L$ of $\R^n$
generated by $L_1 \cup L_2$ is discrete, i.e. a lattice containing
both of the $L_i$. Therefore $\Lambda(L, W_1 \cup W_2)$ contains
$\Lambda_1 \cup \Lambda_2$ and the result follows. 
\end{proof}
}

\begin{cor}\label{cor:from_witte2}
Let $n_1, \ldots, n_r$ be integers greater than 2, and suppose that for each $i$ we are given an embedding of $H \cong \ASL_2(\R)$ as in Corollary \ref{cor:from_witte1}. 
Then for any $y = (y_1, \ldots, y_r) \in 
\prod_1^r \mathcal{Y}_{n_i}$, there is a sequence $(h_{\ell})
\subseteq H$ such that $h_{\ell}y_i \to_{\ell \to \infty} \infty$
simultaneously for all $i$. 

Furthermore, 
suppose the linear parts of the $y_i$ are 
$\Q$-irreducible with respect to the physical subspace. Denote by
$\Delta(H)$ the diagonal embedding of $H$ in $\prod
\ASL_{n_i}(\R)$. Then $\overline{\Delta(H)y}$ is invariant under a
subgroup $S$ in the full group of translations $\prod \R^{n_i},$
which projects onto each of the factors  
$\R^{n_i}$. 
\end{cor}

\begin{proof}
The first assertion is immediate from Corollary \ref{cor:from_witte1}
and Proposition \ref{prop:general_fact}. 
For the second assertion, 
let $T^{(i)}$ be the
unipotent radical (i.e. translational part) of $\ASL_{n_i}(\R)$ and
let $T = T^{(1)} \times \cdots \times T^{(r)}$ be the unipotent
radical of $\prod \ASL_{n_i}$, let $S \subset \Delta(H)$ be the
diagonal embedding of the unipotent radical (translational part) of
$H$, and let $\Omega =
\overline{\Delta(H)y}$. We know that $\Omega$ is invariant under $S$
and our goal is to show that it is invariant under a group $S'$ which
projects onto each $T^{(i)}$. As in the proof of Corollary
\ref{cor:all_translations}, since $T$ is normal in $\prod \ASL_{n_i}$,
it suffices to show that $\Omega_0 = \overline{Sy}$ is 
equal to $S'y$ where $S'\subset T$ is a linear subspace which projects
onto each $T^{(i)}$. Since $Ty$ is a torus of dimension $\sum n_i$,
there is a linear subspace $S'$ such that $\overline{Sy} = S'y$, 
$S'$ is defined over $\Q$ and its projection to each
$T^{(i)}$ is defined over $\Q$ and contains the physical subspace. As
in the proof of Corollary \ref{cor:all_translations}, this implies
that the projection is onto $T^{(i)}$.  
\end{proof}

\begin{proof}[Proof of Theorem \ref{thm:Ratner,cut-and-project}]
To make the idea more transparent we will first prove that one cut and
project set is not a DDanzer set, under the assumptions that it is in
$\R^2$ and is irreducible. We will then proceed to the general
case. 

So let $\Lambda$ be a cut-and-project set
with $\Lambda \subseteq \R^2$
and $\Lambda =\Lambda(L,W)$, where $L \subseteq \R^n$ is a lattice,
$\R^n = \R^2\oplus \R^{n-2}$ is the decomposition of $\R^n$ into the
physical and internal spaces, and $L$ is $\Q$-irreducible with respect
to the physical space.  We want to show that $\Lambda$ is not DDanzer.

Let $H = \ASL_2(\R)$. Applying Proposition \ref{prop:dynamical} we
need to show that for any $T>0$ there is $h \in H$ such that 
\begin{equation}\label{eq:as_required}
h\Lambda \cap B(0,T) = \varnothing.
\end{equation} 
Inspired by \cite{MS13}, we will use the observation that the
projection $\pi_{phys}$ is equivariant with respect to the $H$-action
on $\mathcal{Y}_n$. More precisely, let $L = g_0 \Z^n$. By rescaling
there is no loss of generality in assuming that $\det(g_0)=1$, so we
can regard $L$ as an element of  
$\mathcal{Y}_n$. Consider the embedding of $H$ in $\ASL_n(\R)$ as in
Corollary \ref{cor:from_witte1}. Then $H$ acts simultaneously on
subsets of $\R^2$ via its affine action, and on $\mathcal{Y}_n$ by
left translations, and since the $H$-action is trivial on $\R^{n-2}$,
we have  
\[h\Lambda(L, W) = \Lambda(hL, W) \ \ \forall h \in H, \, L \in \mathcal{Y}_n.\] 
According to Corollaries \ref{cor:from_witte1} and \ref{cor:all_translations}, the orbit-closure
$\overline{HL}$ is non-compact and invariant under translations in
$\R^n$. According to 
Minkowski's theory of successive minima (see
e.g. \cite[Chap. 1]{Cassels}), this implies that in $\overline{HL}$ we
can find grids whose corresponding lattices have arbitrarily large
$n$-th successive minimum. That is, for any $T'$ we 
can find $h \in H$ such that the points of $hL$ are contained in
parallel affine hyperplanes with distance at least $T'$ apart.  
 
Given $T>0$, let $C$ be an open bounded subset of $\R^n$ which
contains the closure of $\pi_{phys}^{-1}(B(0,T)) \cap
\pi_{int}^{-1}(W),$ and let $T'$ be the diameter of $C$. Since
$\overline{HL}$ is not bounded, there is $h' \in H$ so that $h'L$
misses a translate of $C$, and since $\overline{HL}$ is invariant
under translations, there is $h \in H$ so that $hL$ misses $C$. This
implies (\ref{eq:as_required}).    

We now prove the general case of the theorem, i.e. when $\Lambda
=\Lambda_1 \cup \cdots \cup \Lambda_r$, $\Lambda_i =
\Lambda(L_i,W_i)$. By Propositions \ref{prop:reduction_d=2} and \ref{prop:Q_irreducible}
there is no loss of generality in assuming that $d=2$ and each $\Lambda_i$
is irreducible. For each $T>0$ we will find $h \in H$ such that  
\begin{equation}\label{eq:as_required_2}
h\Lambda_i \cap B(0,T) = \varnothing, \ i=1, \ldots, r.
\end{equation} 
We denote the
dimension of the total space (sum of physical space and internal
space) of $\Lambda_i$ by $n_i$, denote the corresponding projections
by $\pi^{(i)}_{phys}, \pi^{(i)}_{int}$, write $n =n_1+ \cdots + n_r$,
and consider the orbit of   
\[ L \df L_1 \oplus \cdots \oplus L_r\] 
under the diagonal action of $\Delta(H)$ on the space of products of grids
$\mathcal{Y}^{(1)} \times \cdots \times  \mathcal{Y}^{(r)}$
(where we are simplifying the notation by 
writing $\mathcal{Y}^{(i)}$ for $\mathcal{Y}_{n_i}$ and
$\mathcal{X}^{(i)}$ for $\mathcal{X}_{n_i}$). Denote by 
$\bar{L}, \bar{L}_i$ the corresponding lattices in $\mathcal{X}_n, \mathcal{X}^{(i)}$.  
By Corollary \ref{cor:from_witte2}, there is $(h'_j) \subseteq H$ such
that $h'_j \bar{L}_i \xrightarrow{j \to \infty} \infty$ for all
$i$. This implies that there are hyperplanes $Q(i,j)$ in $\R^{n_i}$
such that the points in $h'_j \bar{L}_i$ are contained in a union of
translates of $Q(i,j)$ and the distance between the translates tends
to infinity with $j$. Since the space of hyperplanes is compact we may
pass to subsequences to assume that $Q(i,j) \to Q_i$ for each
$i$. Applying Proposition \ref{prop:similar_for_subspaces} we may
replace $h'_j$ with $h_j$ so that $h_jL_i$ and $h'_jL_i$ differ by a
translation 
whose component in the direction perpendicular to $Q_i$ goes to
infinity with $j$. In particular, the grids $h_j L_i$ do not contain
points in balls $B(0,T_j) \subseteq \R^n$ with $T_j \xrightarrow{j \to
  \infty} \infty$.  

In particular, given $T>0$, let $C_i$ be a cube in $\R^{n_i}$ which
contains $\left(\pi_{phys}^{(i)} \right)^{-1} (B(0,T)) \cap 
\left(\pi_{int}^{(i)} \right)^{-1}(W_i)$.
The above discussion ensures that there is $h \in H$ such that $hL_i$
is disjoint from $C_i$ for all $i =1, \ldots,r$. In light of
Proposition \ref{prop:dynamical}, this shows that $\Lambda$ is not DDanzer.  
\end{proof}

\section{Construction of a dense forest}\label{sec:dense_forest}
Another question related to the Danzer problem is the
following. A man stands at an arbitrary point $x$ in a forest with
trunks of radii $\varepsilon>0$, how far can he see? 

\begin{definition}
We say that $Y\subseteq\R^d$ is a \emph{dense forest} if there is a
function $\vre= \vre(T)$ with $\vre(T) \xrightarrow{T \to \infty} 0$,
such that for any $x\in \R^d$ and any $v \in \mathbb{S}^1$ there is $t
\in [0,T]$ and $y\in Y$ such that $\| x +tv - y\| < \vre$.  
\end{definition} 
Note that a Danzer set is a dense forest with $\vre(T)  =
O(T^{-1/(d-1)})$. A dense forest in $\R^2$ satisfying
$\varepsilon=O(T^{-4})$ was constructed in a paper of Chris Bishop,
see \cite[Lemma 2.4]{Bishop}, following a suggestion of Yuval
Peres. However the set appearing in \cite{Bishop} was not uniformly
discrete (it was the union of two uniformly discrete sets). In this
section we use homogeneous dynamics and Ratner's theorem on unipotent
flows to construct uniformly discrete
dense forests in $\R^d$, for arbitrary $d$.  

Suppose $G$ is a Lie group acting smoothly and with discrete
stabilizers on a compact manifold $X$. Then $\mathcal{S} \subseteq X$ is
called a \emph{cross-section} if: 
\begin{itemize}
\item
$\mathcal{S}$ is the image under a smooth and injective map of a bounded domain in $\R^k$,
where $k = \dim X - \dim G$. By smooth and injective we mean that the map 
extends smoothly and injectively to an open set containing the closure
of its domain. 
\item
 There is a neighborhood $B$ of $e$ in $G$ such that the map 
\begin{equation}\label{eq:cross-section}
B \times \mathcal{S} \to X, \ \ \ (b,s) \mapsto b.s 
\end{equation}
is injective and has an open image. 
\item
For any $x \in X$ there is $g \in G$ such that $gx \in \mathcal{S}$. 
\end{itemize}
It follows easily from the implicit function theorem that cross-sections always exist. 

If $G$ is a Lie group, we say that its action on $X$ is
\emph{completely uniquely ergodic} if there is a measure $\mu$ of full
support on $X$ with the property that for any one-parameter subgroup
$H \subseteq G$, $\mu$ is the unique $H$-invariant Borel probability
measure on $X$.  We have the following source of examples for
completely uniquely ergodic actions.  Here we rely on Ratner's theorem
on unipotent flows on homogeneous spaces, see \cite{Ratner91} or
\cite{morris book}.

\begin{prop}\label{prop:examples_of_ue}
Suppose $G$ is a simple Lie group and $\Gamma$ is an arithmetic
cocompact lattice, arising from a $\Q$-structure on $G$ for which $G$
has no proper $\Q$-subgroups generated by unipotent elements, and let
$U$ be a unipotent subgroup of $G$. Then the action of $U$ on
$G/\Gamma$ is completely uniquely 
ergodic. 
\end{prop}

\begin{proof}
Any one-parameter subgroup $U_0$ of $U$ is unipotent, so by Ratner's
theorem, any $U_0$-invariant ergodic measure arises from Haar measure
on an intermediate subgroup $U_0 \subseteq L \subseteq G$ which is
generated by unipotents and intersects a conjugate of $\Gamma$ in a
lattice in $L$. This implies that up to conjugation, $L$ is defined
over $\Q$, so by hypothesis $L=G$ and the only $U_0$-invariant measure
is the globally supported measure on $G/\Gamma$.  
\end{proof}

For examples of groups $G, \Gamma$ satisfying the hypotheses of Proposition
\ref{prop:examples_of_ue}, see \cite{GG09}. In particular the
hypotheses are satisfied for $\Q$-structures on $G=\SL_n(\R)$ for
$n$ prime (\cite[Prop. 4.1]{GG09}). Since the restriction of a
completely uniquely ergodic action of a group $H$ to any proper
subgroup remains completely uniquely ergodic, this furnishes examples
of completely uniquely ergodic actions of $\R^d$ for any $d$.  

\begin{proof}[Proof of Theorem \ref{thm:dense_forest}] 
By Proposition \ref{prop:examples_of_ue} and the preceding remark, the
last assertion in the theorem follows from the first one. 
We will use additive notation for the group operations on $U$, and we
will let $\| \cdot \|$ denote the Euclidean norm on $U$. We first
prove uniform discreteness, i.e. 
$$\inf_{u_1, u_2 \in \mathcal{D}, u_1\neq u_2} \|u_1-u_2\| >0.$$ 
Let $B$ be as in (\ref{eq:cross-section}),
and choose $r>0$ so that $B(0,r)\subseteq B$. If $u_1.x_0, u_2.x_0 \in
\mathcal{S}$ then $u_2 - u_1$ maps a point of $\mathcal{S}$ to
$\mathcal{S}$, hence, by the injectivity of the map
(\ref{eq:cross-section})  cannot be in $B$. In particular $\|u_2 -
u_1\| \geq r$.   

For $v \in \mathbb{S}^{d-1}$ let $C_v(\vre, T)$ be the cylindrical set
which is 
the image of $[0,T] \times \{z\in \R^{d-1} : \|z\| \leq \vre\}$ under
an orthogonal linear transformation that maps the first standard basis
vector $e_1$ to $v$. If $Y$ is not a dense forest then there is a
sequence $T_n \to\infty$, $\vre>0$ and sequences $x_n \in U, v_n \in
\mathbb{S}^{d-1}$ such that for all $u\in \mathcal{C}_n$,  $x_n+u
\notin \mathcal{D}$ where $\mathcal{C}_n \df C_{v_n}(\vre, T_n)$; that
is, for all $u \in \mathcal{C}_n -x_n$, $u.x_0 \notin
\mathcal{S}$. Now let $\mathcal{C}'_n \df C_{v_n}(\vre/2, T_n)$ (so
that the 
$\mathcal{C}'_n$ are parallel to the $\mathcal{C}_n$ but twice as
small in the directions transverse to $v_n$) and define Borel
probability measures  $\nu_n$ on $X$ by  
\[ \int \varphi \, d\nu_n \df \frac{1}{\Vol(\mathcal{C}'_n)}
\int_{\mathcal{C}'_n-x_n} \varphi(u.x_0) du , \ \ \ \varphi \in C_c(X),\]
where $du$ is the Lebesgue measure element on $U$. We claim that
$\nu_n \to \mu$ in the weak-* topology.   

It suffices to show that any accumulation point of $(\nu_n)$ is
$\mu$, and hence to show that any subsequence of $(\nu_n)$ contains a
subsequence converging to $\mu$. By compactness of the space of
probability measures on a compact metric space, after passing to a
subsequence we have $\nu_n \to \nu$, and we need to show that
$\nu=\mu$. Recall that $v_n$ is the direction of the long axis of
$\mathcal{C}'_n$.  
Passing to another subsequence, the vectors $v_n$ converge to a limit
$w$. By complete unique ergodicity, it suffices to show that $\nu$ is
invariant under the one-parameter subgroup $H \df \spa (w)$; but this
is a standard exercise, see e.g. \cite[Proof of Theorem 2]{DM}.  This
proves the claim.  

In order to derive a contradiction, let $B_1$ be the ball $B(0,
\vre/2)$ in $U$, and let $B'$ denote the image of $B_1 \times
\mathcal{S}$ under the map (\ref{eq:cross-section}). With no loss of
generality we can assume that $B_1 $ is contained in the neighborhood $
B$ appearing in the definition of a section, so that this map is
injective on $B_1 \times \mathcal{S}$. Now let
$\varphi$ be a non-negative function supported on $B'$ and with $\int
\varphi \, d\mu>0$. That is, for any $z \in \supp \, \varphi$ there is
$u \in U, \|u\| < \vre/2$ such that $u.z \in \mathcal{S}$.  By
definition of $\mathcal{D}$ we have that for any $u \in \mathcal{C}_n
 -x_n$, $u.x_0 \notin \mathcal{S}$. This implies that for all $u \in
\mathcal{C}'_n -x_n, $  $u.x_0 \notin \supp \, 
\varphi$. This violates $\nu_n \to \mu$. 
\end{proof}



\section{An Equivalent Combinatorial Question}\label{sec:equi_combinatoric_question}

\subsection{Preliminaries}\label{subsec:Comb_Basics}
We recall some standard notions in combinatorial and computational geometry. For a
comrehensive introduction to the notions used in this section
we refer to \cite[\S14.4]{AS08}, 
\cite[\S10]{Mato95}, and \cite{VC71}. 
 
\begin{definition}\label{def:range_space}
A \emph{range space} is a pair $(X,\RA)$ where $X$ is a set, and
$\RA\subseteq\mathcal{P}(X)$ is a collection of subsets of $X$. The elements of
$X$ are called \emph{points}, and the elements of $\RA$ are called
\emph{ranges}.  
\end{definition}

In the literature this is also referred to as a `set system'
or a `hypergraph', where in the latter case $X$ is the set of vertices,
and $\RA$ is the set of hyperedges. 

Many of the commonly studied examples are geometric. For example, $X$
is $\R^d$ or $[0,1]^d$, and $\RA$ is the set of all geometric figures
of some type, such as half spaces, triangles, aligned boxes, convex
sets, etc. For a subset $A\subseteq X$ we denote by $\RA|_A=\{S\cap
A:S\in\RA\}$, the \emph{projection of $\RA$ on $A$}. This notion
allows to consider geometric ranges when $X$ is a discrete set, like a
thin grid in $[0,1]^d$. 

\begin{definition}\label{def:eps_net}
Let $(X,\RA)$ be a range space with $\#X=n$. For a given
$\varepsilon>0$, a set $N_\varepsilon\subseteq X$ is called an
\emph{$\varepsilon$-net} if for every range $S\in\RA$ with $\#(S\cap
X)\ge\varepsilon n$ we have $S\cap N\neq\varnothing$. A similar
definition is made for an infinite set $X$, equipped with a
probability measure $\mu$ on $X$. In that settings $N_\varepsilon$ is
an $\varepsilon$-net if $S\cap N\neq\varnothing$ for every range $S$
with $\mu(S\cap X)\ge\varepsilon$. 
\end{definition}

Notice that the notion of an $\varepsilon$-net resembles the notion of
a Danzer set, when $X=[0,1]^d$ with the standard Lebesgue measure, and
$\RA$ is the set of convex subsets of $X$. Below is the computational geometry version of
the Danzer problem.  It is sometimes referred to as the `Danzer-Rogers
question'.

\begin{question}\label{ques:CG_Danzer}
What is the minimal cardinality of an $\varepsilon$-net $N\subseteq
[0,1]^d$, where $\RA$ is the collection of convex subsets of $X=[0,1]^d$?
Do $O(1/\varepsilon)$ points suffice? 
\end{question}

In this question one may equivalently take $\RA$ to be the collection
of boxes in $X$ or  the collection of ellipsoids in $X$. This is due
to the following: 

\begin{prop}\label{claim:convex_vs_boxes}
For any convex set $K\subseteq\R^d$ there exists boxes $R_1\subseteq
K\subseteq R_2$ with $\Vol(R_2)/\Vol(R_1)\le\alpha_d$, where
$\alpha_d=(3d)^d$.  
\end{prop}

\begin{proof}
The claim follows from John's Theorem, see 
\cite[Lecture 3]{Ba97}.
\end{proof}

\subsection{Proof of Theorem \ref{thm:Equivalent_Questions}}
We divide the
proof of Theorem \ref{thm:Equivalent_Questions} into several
parts. We begin with the more difficult implication $(ii) \implies
(i)$. Let $\|\cdot\|_2, \|\cdot \|_{\infty}$ denote the Euclidean and
sup-norm respectively, let $D$ be the Euclidean metric, and let
$D(x,A)= \sup_{a \in A} D(x,a)$. Let 
\begin{equation}\label{eq:box_and_ball}
\begin{matrix}
& Q_t=\{x\in\R^d:\norm{x}_\infty\le t\}, \\ \\
& B_t=\{x\in\R^d:\norm{x}_2\le t\},
\end{matrix}
\end{equation} 
and consider the following partition of
$\R^d$ into cubical layers that grow exponentially: 
\begin{equation}\label{eq:layers}
L_1=Q_2, \ \ 
L_i=Q_{2^i} \sm Q_{2^{i-1}}, \, i\geq 2. 
\end{equation}
\[\xygraph{
!{<0cm,0cm>;<0.12cm,0cm>:<0cm,0.12cm>::}
!{(2,2)}*{}="a1"
!{(2,-2)}*{}="a2"
!{(-2,-2)}*{}="a3"
!{(-2,2)}*{}="a4"
!{(4,4)}*{}="b1"
!{(4,-4)}*{}="b2"
!{(-4,-4)}*{}="b3"
!{(-4,4)}*{}="b4"
!{(8,8)}*{}="c1"
!{(8,-8)}*{}="c2"
!{(-8,-8)}*{}="c3"
!{(-8,8)}*{}="c4"
!{(16,16)}*{}="d1"
!{(16,-16)}*{}="d2"
!{(-16,-16)}*{}="d3"
!{(-16,16)}*{}="d4"
!{(32,32)}*{}="e1"
!{(32,-32)}*{}="e2"
!{(-32,-32)}*{}="e3"
!{(-32,32)}*{}="e4"
!{(38,0)}*{}="yr"
!{(-38,0)}*{}="yl"
!{(0,38)}*{}="xu"
!{(0,-38)}*{}="xd"
!{(37,3)}*{\cdots}="l"
!{(-37,-3)}*{\cdots}="r"
!{(3,37)}*{\vdots}="u"
!{(-3,-37)}*{\vdots}="d"
!{(-0.7,0.7)}*{_{_{L_1}}}="L1"
!{(-2.7,2.7)}*{_{_{L_2}}}="L2"
!{(-6,6)}*{L_3}="L3"
!{(-13,13)}*{L_4}="L4"
!{(-28,28)}*{L_5}="L5"
!{(2.4,2.8)}*{_{_{(2,2)}}}="0"
!{(4.4,4.8)}*{_{(4,4)}}="1"
!{(8.4,8.8)}*{_{(8,8)}}="2"
!{(16.4,16.8)}*{_{(16,16)}}="3"
!{(32.4,32.8)}*{_{(32,32)}}="4"
"a1"-"a2" "a2"-"a3" "a3"-"a4" "a4"-"a1"
"b1"-"b2" "b2"-"b3" "b3"-"b4" "b4"-"b1"
"c1"-"c2" "c2"-"c3" "c3"-"c4" "c4"-"c1"
"d1"-"d2" "d2"-"d3" "d3"-"d4" "d4"-"d1"
"e1"-"e2" "e2"-"e3" "e3"-"e4" "e4"-"e1"
"xd":"xu" "yl":"yr"
}\] 
Set
\begin{equation}\label{eq:C_d}
C_d=\frac{1}{4d\log_2(10d)}.
\end{equation}

\begin{prop}\label{prop:Comb_Q_implies_Danzer_Q}
Suppose that for every $i$ we have a set $N_i\subseteq L_i$ that
intersects every convex set of volume $C_d$ that is contained in
$L_i$. Then 
\begin{itemize}
\item[(i)] 
$Y \df \bigcup_{i=0}^\infty N_i$ is a Danzer set in $\R^d$.
\item[(ii)] 
If for every $i$ we have $\#N_i=O(g(2^i))$, then $Y$ has growth rate $O(g(T))$.
\end{itemize}
\end{prop}
The proof relies on the following two lemmas.
 
\begin{lem}\label{lem:box_is_contained_in_cube}
Let $R\subseteq\R^d$ be a box. Suppose that $\Vol(Q_t\cap
R)\ge\frac{1}{2}\Vol(R)$, then $R\subseteq Q_{5td}$.  
\end{lem}

\begin{proof}
Let $x_0$ be a vertex of $R$, and let $r_1,\ldots,r_d$ be the $d$
edges of $R$ with one end-point at $x_0$. Denote by $\absolute{r}$ the
length of a segment $r$, then we have
$\Vol(R)=\prod_{i=1}^d\absolute{r_i}$.  

Set $K=Q_t\cap R$ and for every $i\in\{1,\ldots,d\}$ fix
\[k_i=\text{a segment of maximal length in }K,\text{ which is parallel to }r_i.\]

Clearly, for every $i\in\{1,\ldots,d\}$ we have $\Vol(K)\le
\absolute{k_i}\cdot\prod_{j\neq i}\absolute{r_j}$. Hence the
assumption $\frac{1}{2}\Vol(R)\le \Vol(K)$ implies that  
\begin{equation}\label{eq:1/2r_i<k_i}
\frac{1}{2}\absolute{r_i}\le\absolute{k_i}
\end{equation} 
for all $i\in\{1,\ldots,d\}$. 

Let $\ell=\diam(R)$, and let $k\in K$. Since $D(0,k) \le t\sqrt{d}$, we have  
\[R\subseteq B(k,\ell)\subseteq B_{t\sqrt{d}+\ell}\subseteq Q_{t\sqrt{d}+\ell}.\]

On the other hand 
\[
\begin{split}
\ell & =\diam(R)=\sqrt{\absolute{r_1}^2+\ldots+\absolute{r_d}^2}
\le\sqrt{d}\max_i\{\absolute{r_i}\} \\ & \stackrel{(\ref{eq:1/2r_i<k_i})}
\le 2\sqrt{d}\max_i\{\absolute{k_i}\}\le 4td.
\end{split}
\] 

So $R\subseteq Q_{5td}$. 
\end{proof}

\begin{lem}\label{lem:Layers_main_idea}
For any box $R$ of volume $1$ in $\R^d$ there is a layer $L_i$ such
that $L_i\cap R$ contains a convex set $K$ with $\Vol(K)\ge C_d$, where
$C_d$ is as in (\ref{eq:C_d}). 
\end{lem} 

\begin{proof}
Let $m\in\N$ be the minimal integer such that
$R\subseteq\bigcup_{i=0}^mL_i=Q_{2^m}$. Let $j\in\N$ be the minimal
integer satisfying $5d\le 2^j$. So we may also write  
\[Q_{2^m}=Q_{2^{m-j-1}}\cup L_{m-j}\cup L_{m-j+1}\cup\cdots\cup L_m.\]

Since $\Vol(R)=1$ we either have $\Vol(Q_{2^{m-j-1}}\cap
R)\ge\frac{1}{2}$ or $\Vol((L_{m-j}\cup\cdots\cup L_m)\cap
R)\ge\frac{1}{2}$. If $\Vol(Q_{2^{m-j-1}}\cap R)\ge\frac{1}{2}$, then
by Lemma \ref{lem:box_is_contained_in_cube} we have  
\[R\subseteq Q_{2^{m-j-1}\cdot 5d}\subseteq Q_{2^{m-1}}=\bigcup_{i=0}^{m-1}L_i,\] 
contradicting the minimality of $m$. So $\Vol((L_{m-j}\cup\ldots\cup
L_m)\cap R)\ge\frac{1}{2}$, and therefore $\Vol(L_i\cap
R)\ge\frac{1}{2(j+1)}\ge\frac{1}{2\log_2(10d)}$ for some
$i\in\{m-j,m-j+1,\ldots,m\}$.    

It remains to find a convex set $K\subseteq
L_i\cap R$ with $\Vol(K)\ge C_d=1/[4d\log_2(10d)]$. Denote by
$F_1,\ldots,F_{2d}$ the external $d-1$-dimensional faces of $L_i$
(namely, the faces of the cube $Q_{2^i}$). Each $F_i$ defines a convex
set  
\[K_i=\{x\in L_i: \forall j\neq i, D(x,F_i)\le D(x,F_j)\}.\] 
The sets $K_i\cap R$ are convex and one of them contains at least
$(2d)^{-1}$ of the volume of $L_i\cap R$, which gives the desired $K$.   
\end{proof}

\begin{proof}[Proof of Proposition \ref{prop:Comb_Q_implies_Danzer_Q}]
Assertion $(i)$ follows directly from Lemma \ref{lem:Layers_main_idea}. As for
$(ii)$, by assumption there is a constant $C_1$
such that $\#N_i\le C_1g(2^i)$. By adding points to some of the
$N_i$'s we may assume that $g$ attains integer values, and that
$\#N_i=C_1g(2^i)$ for every $i$. Recall that $g(x)=\Omega(x^d)$
(where, as usual, $h_1(x) = \Omega(h_2(x))$ means that $\liminf
\frac{h_1(x)}{h_2(x)} >0$). Thus by adding more points if needed and
increasing $C_1$ accordingly, we may also assume that the function
$\frac{g(x)}{x^d}$ is non-decreasing. 

For a measurable set $A$ we denote by $\mathfrak{D}(A)=\frac{\#(Y\cap
  A)}{\Vol(A)}$, the density of the set $Y$ in $A$, where
$Y=\bigcup_iN_i$. Note that for every $i>0$ the layer $L_i$ is the
union of $4^d-2^d$ cubes of edge length $2^{i-1}$, that intersect only
at their boundaries. So for every $i>0$ we have 
\[\mathfrak{D}(L_i)=\frac{C_1g(2^i)}{(4^d-2^d)(2^{i-1})^d}=
\frac{C_1}{2^d-1}\cdot\frac{g(2^i)}{2^i}.\] 

Since $\frac{g(x)}{x^d}$ is non-decreasing,
$\mathfrak{D}(L_i)\ge\mathfrak{D}(L_{i-1})$, and therefore
$\mathfrak{D}(L_i)\ge\mathfrak{D}(Q_{2^{i-1}})$. Also note that for
every $i>0$ we have $\Vol(L_i)=(2^d-1)\cdot \Vol(Q_{2^{i-1}})$, then 
\[\#N_i=\mathfrak{D}(L_i)\cdot \Vol(L_i)\ge\mathfrak{D}(Q_{2^{i-1}})\cdot (2^d-1)\Vol(Q_{2^{i-1}})=(2^d-1)\#(Y\cap Q_{2^{i-1}}).\]

In particular, for every $i$ we have $\#(Y\cap Q_{2^i})\le 2\#(N_i)=2C_1g(2^i)$.
Then for a given $n$, let $i\in\N$ be such that $n\le 2^i<2n$. Then
\[\#(Y\cap Q_n)\le\#(Y\cap Q_{2^i})\le 2C_1g(2^i)\le 2C_1g(2n).\]
\end{proof}

\begin{proof}[Proof of Theorem \ref{thm:Equivalent_Questions}]
For
$(ii) \implies (i)$, let $\varepsilon_i=\alpha_d^{-1}C_d\cdot
2^{-di}$, where $C_d$ is as in (\ref{eq:C_d}), and $\alpha_d=(3d)^d$
is as in Proposition \ref{claim:convex_vs_boxes}. Let $N_i''$ be an
$\varepsilon_i$-net for $(X=[-1,1]^d,\{\text{boxes}\})$ with
$\#N_i''\le Cg\left(\varepsilon_i^{-1/d}\right)$. Rescale by a factor of $2^i$ in
each axis. So $X$ becomes $Q_{2^i}$, and $N_i''$ becomes
$N_i'\subseteq Q_{2^i}$, a set that intersects every box of volume
$\varepsilon_i\cdot 2^{di}=\alpha_d^{-1}C_d$ in $Q_{2^i}$, with
$\#N_i'\le Cg(\varepsilon_i^{-1/d})=Cg((\alpha_d^{-1}C_d)^{-1/d}\cdot
2^i)$. Note that since $g(x)$ has polynomial growth we have
$\#N_i'=O(g(2^i))$ (with a uniform constant for all $i$), and it
follows from Proposition \ref{claim:convex_vs_boxes} that $N_i'$ intersects
every convex set of volume $C_d$ in $Q_{2^i}$. Let $N_i=N_i'\cap L_i$,
where $L_i$ is as in (\ref{eq:layers}). Then $\#N_i=O(g(2^i))$ and
$N_i$ intersects every convex set of volume $C_d$ that is contained in
$L_i$. By Proposition \ref{prop:Comb_Q_implies_Danzer_Q} the set
$Y=\bigcup_iN_i\subseteq \R^d$ is a Danzer set with growth rate
$O(g(T))$.  

It remains to prove the easier direction $(i) \implies (ii).$ 
Suppose that $Y\subseteq\R^d$ intersects every box of volume $1$ in
$\R^d$. For a given $\varepsilon>0$ consider the square
$Q_\varepsilon$ of edge length $\varepsilon^{-1/d}$, centered at the
origin. Then $N_\varepsilon \df Y\cap Q_\varepsilon$ intersects every
box of volume $1$ that is contained in $Q_\varepsilon$. Contract
$Q_\varepsilon$ by a factor of $\varepsilon^{1/d}$ in every one of the
axes. Then $\varepsilon^{1/d}
Q_\varepsilon=[-\frac{1}{2},\frac{1}{2}]^d$, and
$\varepsilon^{1/d}N_\varepsilon$ intersects every box of volume
$\varepsilon$ in it. In addition, if $Y=O(g(T))$, then there exists a
constant $C$ such that for every $\varepsilon>0$ we have
$\#\varepsilon^{1/d}N_\varepsilon=\#N_\varepsilon\le
Cg(\varepsilon^{-1/d})$. 
\end{proof}

\begin{proof}[Proof of Corollary \ref{cor:rational_Danzer}]
Given $D\subseteq\R^d$ with growth rate $O(g(T))$, that intersects
every convex set $K\subseteq\R^d$ of volume $1$, setting
$D'=\beta_d\cdot D$ for a suitable constant $\beta_d$, that depends
only on $d$, we obtain a set with the same growth rate that intersects
every convex set of volume $C_d$. 
Let $A_i=D'\cap L_i$, then $\#A_i=O(g(2^i))$, and it intersects every convex set $K\subseteq L_i$ of volume $C_d$ . Notice that a convex in $L_i$ with volume $C_d$
must contain a rectangle with some fixed thickness. So taking a thin
enough rational grid $\Gamma_i$ in $L_i$, and replacing every $x\in A_i$ by the
$2^d$ vertices of the minimal cube with vertices in $\Gamma_i$ that
contains $x$, we obtain a set $N_i\subseteq L_i\cap\Gamma_i\subseteq
\Q^d$ with the same properties as $A_i$. So by Proposition
\ref{prop:Comb_Q_implies_Danzer_Q}, $D_\Q\df\bigcup_i N_i$ is as
required.    
\end{proof}


\section{An improvement of a construction of Bambah and
  Woods}\label{sec:T^dlogT_construction} 
As we saw in Theorem \ref{thm:Equivalent_Questions}, the existence
of Danzer sets with various growth rates is equivalent to the
existence of $\vre$-nets for the range space of boxes. 
Finding bounds on the cardinalities of $\varepsilon$-nets in range spaces
is an active topic of research in combinatorics and computational
geometry.  We now derive Theorem \ref{thm:T^dlogT} from results in computational
geometry. 
Many of the results in this direction utilize the low complexity of the range space, which
is measured using the following notion. 
 
\begin{definition}
Let $(X,\RA)$ be a range space. A finite set $F\subseteq X$ is called \emph{shattered} if 
\[\#\{F\cap S:S\in\RA\}=2^{\#F}.\] 
The \emph{Vapnic Chervonenkis dimension}, or 
\emph{VC-dimension}, of a range space $(X,\RA)$ is 
\[VCdim(X,\RA)=\sup\{\#F:F\subseteq X \text{ is shattered}\}.\]
\end{definition}


\begin{example}\label{Examples:VCdim}
To explain this notion we compute the VC-dimension for the
following two simple examples, where $X=[0,1]^d$: 
\begin{itemize}
\item 
$\RA$ is the set of closed half-spaces, where $H$ is a half-space, i.e.
$H=\{x\in\R^d:f(x)\le t\}$, for some linear functional $f$ and
$t\in\R$. We show that $VCdim(X,\RA)=d+1$. First note that if
$\Lambda\subseteq X$, $\#\Lambda=d+1$, and $\Lambda$ is in general
position, then $\Lambda$ is shattered. On the other hand, by Radon's
Theorem, 
every $\Lambda\subseteq X$ of size $d+2$
can be divided into two sets $A,B$ such that their convex hulls
intersect. In particular, there is no half-space $H$ such that $A =
\Lambda\cap H$. 

\item 
$\RA$ is the set of convex sets. Here $VCdim(X,\RA)=\infty$: let $C$
be a $d-1$-dimensional sphere in $[0,1]^d$. Then every finite $C_0
\subseteq C$ is shattered since $C\cap \conv(C_0)=C_0$ and $\conv(C_0)\in\RA$. 

\end{itemize}
\end{example}

Low VC-dimension in particular yields a bound on the cardinality of
$\RA$. Lemma \ref{lem:bound_ranges} below was proved originally by
Sauer, and independently by Perles and Shelah; see \cite[Lemma 14.4.1]{AS08}.  

\begin{lem}\label{lem:bound_ranges}
If $(X,\RA)$ is a range space with VC-dimension $d$, and $\#X=n$, then
$\#\RA\le\sum_{i=0}^d\binom{n}{i}$. 
\end{lem}

As a corollary we have (see \cite{AS08} Corollary $14.4.3$): 

\begin{cor}\label{cor:VC_of_intersections}
Let $(X,\RA)$ be a range space of VC-dimension $d$, and let
$\RA_k=\{s_1\cap\cdots\cap s_k:s_i\in\RA\}$. Then $VCdim(X,\RA_k)\le
2dk\log(dk)$. 
\end{cor}

Since every $d$-dimensional box is the intersection of $2d$
half-spaces, combining Example \ref{Examples:VCdim} on half-spaces and
Corollary \ref{cor:VC_of_intersections} we deduce the following.  

\begin{cor}\label{cor:VCdim_boxes}
Let $Q$ be a $d$-dimensional cube, then $VCdim(Q,\{\text{boxes}\})\le
4d(d+1)\log(2d(d+1))$. 
\end{cor}



We proceed to the proof of Theorem \ref{thm:T^dlogT}. 
To simplify notations, we depart slightly from 
\equ{eq:box_and_ball}, and denote by $Q_n\subseteq\R^d$ the cube of
edge length $n$ centered at the origin in this section. We begin with
the following proposition, which is a special case of a result of 
Haussler and Welzl \cite{HW87}. For completeness 
we include the proof of this proposition, which we learned from Saurabh
Ray. 

\begin{prop}\label{prop:probabilistic_constraction}
For any $d$ there is a constant $C$ such that for any integer $n>0$
there exists a finite set $N\subseteq Q_n$ with 
$\#N=Cn^d\log n$, which intersects any box
$R\subseteq Q_n$ of volume $1$.
\end{prop}

\begin{proof}
Let $\Gamma_n$ be the set of vertices of a regular decomposition of $Q_n$ into
cubes of edge-length $1/n$. Then each edge of $Q_n$ is divided into $n^2$
points, and therefore $\#\Gamma_n=n^{2d}$.  
Note that any box $R\subseteq Q_n$ of volume $1$ that is contained in
$Q_n$ contains $\Omega(n^{2d}/n^d)=\Omega(n^d)$ points of $\Gamma_n$
(up to an error of $O(n^{d-1})$), and at least $n^d/2$ points (when
$n$ is sufficiently large). 

Let 
\[p=\frac{c\log(n)}{n^d}\in(0,1),\] 
where $c$ depend only on $d$, and will be chosen later. Let $N$ be a
random subset of $\Gamma_n$ that is obtained by choosing points from
$\Gamma_n$ randomly and independently with probability $p$. Then $\#N$
is a binomial random variable $B(m,p)$, where $m=n^{2d}$, with
expectation  
\[\E(\#N)=mp=n^{2d}\cdot p=c\cdot n^d\log(n).\] 

Since $\#N=B(m,p)$ the values of $\#N$ concentrate near $\E(\#N)$. To
be precise, using the Chernoff bound for example (see \cite{Chernoff})
one obtains  
\[Prob[\absolute{\#N-\E(\#N)}\ge \E(\#N)/2]\le e^{-\frac{\E(\#N)}{16}}.\]

So in particular with probability greater than $(n-1)/n$ we have 
\begin{equation}\label{eq:Chernoff_for_N}
\frac{1}{2}cn^d\log(n)\le\#N\le \frac{3}{2}cn^d\log(n).
\end{equation}

$N$ misses a given box $R$ if all the points in $R$ are not chosen in
the random set $N$. This occurs with probability at most 
\[(1-p)^{n^d/2}=
\left(1-\frac{c\log(n)}{n^d}\right)^{\frac{n^d}{c\log(n)}\cdot\frac{c\log(n)}{2}}\le
\left(1-\frac{c\log(n)}{n^d}\right)^{\left(\frac{n^d}{c\log(n)}+1\right)
  \frac{c\log(n)}{4}}\] 
\[\le e^{-c\log(n)/4}=\frac{1}{n^{c/4}}.\]

Let our collection of ranges $\RA$ be the collection of boxes in $Q_n$ (where
two boxes $R_1, R_2$ are considered to be equal if their intersections with
$\Gamma_n$ coincide).
By Corollary \ref{cor:VCdim_boxes} we have $VCdim(Q_n,\RA
)\le 4d(d+1)\log(2d(d+1))\le 4(d+1)^3\df t$.  
By Lemma \ref{lem:bound_ranges} we have
\[\#\RA\le\sum_{i=0}^t\binom{\#\Gamma_n}{i}\le (t+1)n^{2dt}.\]
Pick $c/4=2dt+1=O(d^4)$. A standard union bound gives that the
probability to miss some box $R$ is at most
\begin{equation}\label{eq:Union_bound}
(t+1)n^{2dt}\cdot\frac{1}{n^{c/4}}=O\left(\frac{1}{n}\right).
\end{equation}

By (\ref{eq:Chernoff_for_N}) and (\ref{eq:Union_bound}) we deduce that
 there exists a set $N\subseteq \Gamma_n$ of size at most
$\frac{3}{2}cn^d\log(n)$ that intersects every box of volume $1$ in
$Q_n$. 
\end{proof}

\begin{proof}[Proof of Theorem \ref{thm:T^dlogT}]
Let $\varepsilon>0$. Let $n\in\N$ be the minimal positive integer that
satisfies $1/n^d\le\varepsilon$. By Proposition
\ref{prop:probabilistic_constraction} for every $n\in\N$ we have a set
$N_n\subseteq Q_n$, with $\#N_n\le Cn^d\log(n)$ that intersects every
box of volume $1$ in $Q_n$. Rescaling by a factor of $1/n$ in each axis,
we obtain a set $Y_n\subseteq[-1/2,1/2]^d$ of size at most
$Cn^d\log(n)$ that intersects every box of volume $1/n^d$ in
$[-1/2,1/2]^d$. In particular, for every $\varepsilon$ we have
constructed an $\varepsilon$-net of cardinality $Cn^d\log(n)$ for the range
space $([0,1]^d,\{\text{boxes}\})$. Notice that
$n-1<\varepsilon^{-1/d}\le n$, so we showed $(ii)$ of Theorem
\ref{thm:Equivalent_Questions}, with $g(x)=x^d\log(x)$, and therefore
we have a Danzer set of growth rate $O(T^d\log(T))$. 
\end{proof}

\section{Some open questions}\label{sec: questions}
We conclude with a list of open questions which would constitute
further progress toward Danzer's question. 

\subsection{Fractal substitution systems and model sets}
In \S \ref{sec:substitution} we showed that Delone sets obtained from {\em polygonal}
substitution tilings are not Danzer sets. There is also a theory of
substitution tilings in which the basic tiles are fractal sets (see
\cite{So97} and the references therein), and our methods do not apply to these tilings. It would be
interesting to extend Theorem \ref{thm:Substitution_Not_Danzer} to substitution
tilings which are not polygonal. Also it would be interesting to
extend  Theorem \ref{thm:Substitution_Not_Danzer} to finite unions of
sets obtained from substitution tilings.

Similarly, in our definition of cut-and-project sets, the internal
space was taken to be a real vector space. More general constructions,
often referred to as {\em model sets}, 
in which the internal space is  an arbitrary locally compact abelian  group have
also been considered, see e.g. \cite{Me94, BM00}. It is likely that
Theorem \ref{thm:Ratner,cut-and-project} can be extended to 
model sets with a similar proof. 

\subsection{Quantifying the density of forests}
 We do not know whether the dense forest constructed in
\S\ref{sec:dense_forest} is a Danzer set. One can also ask how close
it is to being one, in the following sense. One can quantify the `density' of a dense forest
by obtaining upper bounds on the function $\vre(T)$; as we remarked
above a Danzer
set is a dense forest with $\vre(T) = O\left(T^{-1/(d-1)}\right)$. As mentioned in \S\ref{sec:dense_forest}, the example of Peres given in
\cite{Bishop} is a dense forest in $\R^2$ with $\vre(T) = O(T^4)$. It
would be interesting to construct dense forests in the plane with $\vre(T) =
O(T^{s})$ for $s<4$.  

In our example of a dense forest, an upper bound on the function $\vre(T)$
would follow from a bound on the rate of convergence of ergodic
averages in Ratner's equidistribution theorem, for one parameter
unipotent flows on homogeneous spaces such as those in Proposition
\ref{prop:examples_of_ue}. Note that in these examples, when $d>1$,
one-parameter groups are not horospherical and hence such bounds are
very difficult to obtain. In a work in progress \cite{LMM},
Lindenstrauss, Margulis and Mohammadi prove such bounds but they are
much weaker than the bounds required to prove the Danzer property.  

\subsection{Relation to dynamics on pattern spaces}
The collection $\mathscr{Cl}$ of closed subsets of $\R^d$ is compact with respect to  a natural 
topology sometimes referred to as the {\em Chabauty topology}. Roughly
speaking, two $A$ and $B$ are close to each other in $\mathscr{Cl}$ 
if their intersections with large balls are close in the
Hausdorff metric. The group $\ASL_d(\R)$ of volume preserving affine
maps acts on $\mathscr{Cl}$ via its action on $\R^d$, and
recalling Proposition \ref{prop:dynamical}, we can interpret the
Danzer 
property as a statement about this action. Namely $Y \in \mathscr{Cl}$
is not DDanzer (in the sense of \S\ref{sec:cut-and-project}) 
if its orbit-closure contains the empty set. This leads to the
following question:
\begin{question}\label{q: minimal sets}
Is it true that the only minimal sets for the action of $\ASL_d(\R)$
on $\mathscr{Cl}$
are the fixed points $Y=\varnothing, Y= \R^d$? 
\end{question}

In \cite{Gowers}, Gowers proposed a weakening of the Danzer question.
He asked whether there are Danzer sets $Y$  which have the additional
property that $\sup\{\#(Y \cap C) :C \text{ convex}, \, \Vol(C)=1 \}
<\infty.$ It is not hard to show that an affirmative answer to
Question \ref{q: minimal sets} would imply a negative answer to
Gowers' question. 

\end{document}